\documentclass[12pt,oneside]{amsart}
\usepackage{amsmath}
\usepackage{amsthm}
\usepackage{amsfonts}
\usepackage{amssymb}
\usepackage{enumerate}
\newtheorem{theorem}{Theorem}
\newtheorem{lemma}[theorem]{Lemma}
\newtheorem{proposition}[theorem]{Proposition}

\newtheorem{claim}[theorem]{Claim}

\theoremstyle{definition}

\theoremstyle{remark}

\newcommand{\BB}{{\mathbb B}}
\newcommand{\la}{{\lambda}}
\newcommand{\DD}{{\mathbb D}}
\newcommand{\OO}{{\mathcal O}}
\newcommand{\uU}{{\mathcal U}}

\newcommand{\NN}{{\mathbb N}}

\newcommand{\RR}{{\mathbb R}}
\newcommand{\CC}{{\mathbb C}}

\newcommand{\eps}{\varepsilon}

 \DeclareMathOperator{\re}{Re}
 
 \DeclareMathOperator{\dist}{dist}

\DeclareMathOperator{\diag}{diag}

\renewcommand{\phi}{\varphi}

\subjclass[2010]{32E30, 30E05}

\begin{document}

\title{Nevanlinna-Pick interpolation problem in~the ball}

\address{Institute of Mathematics, Faculty of Mathematics and Computer Science, Jagiellonian
University,  \L ojasiewicza 6, 30-348 Krak\'ow, Poland}

\author{\L ukasz Kosi\'nski}\email{lukasz.kosinski@im.uj.edu.pl}
\author{W\l odzimierz Zwonek}\email{wlodzimierz.zwonek@im.uj.edu.pl}
\thanks{The first author was supported by the NCN grant UMO-2014/15/D/ST1/01972.}

\thanks{The second author was  supported by the
OPUS grant no. 2015/17/B/ST1/00996 financed by the National
Science Centre, Poland.}
\keywords{}

\begin{abstract}
We solve a three point Nevanlinna-Pick problem in the Euclidean ball. In particular, we determine a class of rational functions that interpolate this problem.
\end{abstract}
\maketitle
\section{Introduction}

The Nevanlinna-Pick problem for a domain $D$ of $\mathbb C^n$ may be stated as follows: {\it Given distinct points $z_1,\ldots, z_N$ in $D$ and numbers $\la_1,\ldots, \la_N$ in the unit disc $\DD$ decide whether there is an analytic function $F:D\to \DD$ that interpolates, i.e. $F(z_j) = \la_j$, $j=1, \ldots, N$.} The problem is very classical, its original version was stated for $D=\DD$ and solved by Pick in 1916 (see \cite{Pick}).

This problem has been considered in different domains and many attempts have been made to extend it in different directions. In general, the analogue of Nevanlinna-Pick's theorem does not hold for domains other than the disc. So far it is not clear how to get any solvability criterion for an arbitrary domain $D$. An important result towards understanding this problem was achieved by Sarason who found deep relations between the Nevanlinna-Pick problem and several results in operator theory (see \cite{Sar}). Cole, Lewis and Wermer in \cite{Cole} considered the problem for any uniform algebra.
In a sequence of influential papers (see \cite{Agl1, Agl2, Agl3, Agl4} and also a monograph \cite{Aglbook}) Agler and McCarthy used operator theory approach to carry out an analysis of Nevanlinna-Pick problem for the bidisc. However, methods developed there did not work for any other domains. They even failed for higher dimensional polydiscs $\DD^n$. Some results for $\DD^n$ and the Euclidean ball $\BB_n$ were obtained by Hamilton \cite{Ham}. Interpolation in the Euclidean ball was also investigated by Amar and Thomas \cite{Ama, Ama1}. Recently, the first author of this paper found an alternate approach to the Nevanlinna-Pick problem in the polydisc (see \cite{Kos 2015}) and $N=3$ which resulted in solving the problem in this situation. This approach also allowed Knese to prove the von Neumann inequality for $3\times 3$ matrices (see \cite{Kne 2015}). In our paper we adopt the methods from \cite{Kos 2015} to deal with the Nevanlinna-Pick problem for the Euclidean ball.

\bigskip

Roughly speaking we show that a three-point Nevanlinna-Pick problem in the Euclidean ball may be expressed in terms of a dual problem $\DD\to \BB_n$. We also find a class of rational functions of degree at most $2$ interpolating every such problem. Extremal functions in this class are, up to a composition with an automorphism of $\BB_n$, of the form \eqref{class FD} or \eqref{class FND}. The precise statement of the result is postponed to the next section.

\section{Definitions and results}

\subsection{Interpolation problems. Extremal maps}

The definitions we present here are taken from \cite{Kos-Zwo 2015} and \cite{War 2015}.

Let $D$ be a domain in $\mathbb C^n$ and let $N\geq 2$. Fix pairwise distinct points
$\lambda_1,\ldots,\lambda_N\in\DD$ and points $z_1,\ldots,z_N\in D$. We call the interpolation data
\begin{equation}\nonumber
\DD\to D,\quad \lambda_j\mapsto z_j,\; j = 1,\ldots,N,
\end{equation}
 {\it extremally solvable} if it is solvable i.e. there is a map $h\in\OO(\DD,D)$ such
that $h(\lambda_j ) = z_j $, $j = 1,\ldots, N$, and there is no $f$ holomorphic on a neighborhood of
$\overline{\DD}$ with the image in $D$ such that $f(\lambda_j ) = z_j$, $j = 1,\ldots, N$.

Note that the latter condition is equivalent to the fact that there is no $h\in
\OO(\DD,D)$ such that $h(\lambda_j ) = z_j$, $j = 1, \ldots, N$, and $h(\DD)$ is relatively compact in $D$.

This leads us to the following definition (see \cite{Kos-Zwo 2015}). Let $f :\DD\to D$ be an analytic disc. Let $\lambda_1,\ldots,\lambda_N\in D$ be pairwise distinct points. We say that $f$ is a {\it weak $N$-extremal with respect to $\lambda_1, \ldots,\lambda_N$} if the problem
\begin{equation}\nonumber
\DD\to D,\quad \lambda_j\mapsto f(\lambda_j),\; j=1,\ldots,N,
\end{equation}
is extremally solvable.

Naturally, we shall say that $f$ is a {\it weak $N$-extremal} if it is a weak extremal with
respect to some $N$ pairwise distinct points in the unit disc.

The idea of the above definition has roots in \cite{Agl-Lyk-You 2013} where authors introduced the
notion of {\it $N$-extremal maps}, demanding that the above problem is extremal for all choices of
pairwise distinct points $\lambda_1,\ldots,\lambda_N$. Our definition of extremals is weaker; however, for many domains classes of $N$-extremals and weak $N$-extremals coincide (see \cite{Kos-Zwo 2015}). This is the
case for among others homogenous (i.e. with transitive group of holomorphic automorphisms)
and balanced domains. In particular, both definitions are equivalent
for the Euclidean ball. Similar maps were also investigated by Edigarian \cite{Edi 1995}.

The dual problem to the one presented above (we call it the \textit{$N$-Pick problem for $D$}) is to interpolate the following problem
\begin{equation}\nonumber
D\to\DD,\quad z_j\mapsto\lambda_j,\;j=1,\ldots,N,
\end{equation}
i.e. to find an $F\in\OO(D,\DD)$ such that $\lambda_j=F(z_j)$, $j=1,\ldots,N$. The problem is {\it extremal } if there is no $G\in\OO(D,\DD)$ 
with $G(z_j)=\lambda_j$, $j=1,\ldots,N$, and the image $G(D)$ lies relatively compactly in $\DD$.

The two problems present two different generalizations of the classical Nevanlinna-Pick problem in the unit disc which are mutually dual.

In the paper we show a very close relation between them for $D=\mathbb B_n$ and $N=3$.

\subsection{Case $N=2$. Lempert theorem}\label{sec: N=2}
Recall that in the case $N=2$ the mutual relationship between the above two problems is well understood. In the case of the ball and the polydisc the description of extremal problems is obtained very easily. In a much more general case of bounded convex domains the same description of the extremality of the problems
\begin{equation}\label{NP-problem}
D\to\DD,\quad z_j\mapsto\lambda_j,\; j=1,2,
\end{equation}
and 
\begin{equation}\label{extremal-problem}
\DD\to D,\quad \mu_j\mapsto z_j,\; j=1,2,
\end{equation}
is a consequence of the Lempert theorem (see \cite{Lem 1981}). More precisely, for the fixed $z_1,z_2\in D$ the extremality of \eqref{NP-problem} and (\ref{extremal-problem}) implies the existence of the interpolating functions
$f\in\OO(\DD,D)$ and $F\in\OO(D,\DD)$ such that $F\circ f$ is an automorphism of $\DD$ (in the language we shall use, Blaschke products of degree one). 
Moreover, for the given distinct points $z_1,z_2\in D$ the existence of the extremal interpolation function $F$ and corresponding points $\lambda_1,\lambda_2$ in the problem \eqref{NP-problem} (respectively, the existence of the extremal interpolation function $f$ and $\mu_1,\mu_2$ in the problem \eqref{extremal-problem}) implies the existence of $f$ and $\mu_1,\mu_2$ (respectively, $F$ and $\lambda_1,\lambda_2$) as above. One of our aims is to find an analogue to this description in the case of $N=3$ and $D$ being the unit ball $\BB_n$.

\subsection{Complex $N$-geodesics} 

The above observations make us recall the following definition (see \cite{Kos-Zwo 2015}). 
 An analytic disc $f :\DD\to D$ is called a
{\it complex $N$-geodesic} if there is a holomorphic function $F : D\to\DD$ such that $b:=F\circ f$
is a non-constant Blaschke product of degree at most $N-1$.
The function $F$ is called a {\it left inverse} to the complex $N$-geodesic $f$.

Recall that if $f :\DD\to D$ is a complex $N$-geodesic, $F: D\to\DD$ is its left
inverse and $b:=F\circ f$ then for any pairwise distinct
$\lambda_1,\ldots,\lambda_N\in\DD$ the interpolation problems  
\begin{equation}\nonumber
D\to\DD,\quad f(\lambda_j )\mapsto b(\lambda_j ),\; j = 1,\ldots,N,
\end{equation}
and
\begin{equation}\nonumber
\DD\to D,\quad \lambda_j\mapsto f(\lambda_j),\; j=1,\ldots,N,
\end{equation}
are extremal.

\subsection{Solution of the extremal $3$-point problem in the ball}
 
In our paper we deal with $D:=\BB_n$ and $N=3$.
The main result of the paper is the following.

\begin{theorem}\label{theorem-main} Let $F\in\OO(\BB_n,\DD)$ and let $z_1,z_2,z_3\in\BB_n$ be pairwise different. 
A $3$-point Pick problem in $\BB_n$ 
\begin{equation}\nonumber
\BB_n\to\DD,\quad z_j\mapsto\lambda_j:=F(z_j),\; j=1,2,3,
\end{equation}
is extremal if and only if there is an $f\in\OO(\DD,\BB_n)$ such that either $f$ passes through the points  $z_1,z_2,z_3$ 
and $F\circ f$ is a non-constant Blaschke product of degree  at most $2$ or $f$ passes through at least two of the points $z_1,z_2,z_3$ and $F\circ f$ is a Blaschke product of degree $1$.
\end{theorem}

Actually, we prove much more than it is formulated above. Before we may state the more general and more detailed results we need to introduce some notions and recall known facts.

\subsection{Geometry of the unit ball}

To proceed with the presentation of main results let us recall some well-known facts on the unit ball (see e. g. \cite{Jar-Pfl 2013} and \cite{Rud 1980}). First recall that the group of automorphisms of $\BB_n$ is generated by idempotent mappings $\chi_w$, $w\in\BB_n$,  and unitary mappings. For the fixed $w$ the mapping $\chi_w$ is the automorphism interchanging $w$ and $0$, $w\in\BB_n$. In the special case of the unit disc ($n=1$) we put $m_a:=\chi_a$, $a\in\DD$. Then $m_a$ is the idempotent M\"obius map. Recall that the automorphisms of the unit ball map the parts of complex lines lying in the ball onto the same type of sets. Moreover, any three points from the ball may be mapped by some automorphism into a given two-dimensional intersection of the affine subspace with the ball.

We also need to know the effective formula for the Carath\'eodory distance of the unit ball - the uniquely determined holomorphically invariant function (for the definition of the Carath\'eodory distance and its properties see e.g. \cite{Jar-Pfl 2013}):
\begin{equation}\nonumber
c_{\BB_n}^*(w,z)=\sqrt{1-\frac{(1-||w||^2)(1-||z||^2)}{|1-\langle w,z\rangle|^2}},\; w,z\in\BB_n,
\end{equation}
where $\langle\cdot,\cdot\rangle$ denotes the standard inner product in $\CC^n$.

\subsection{Reformulation of the extremal $3$-point Pick problem in the ball}\label{sec: refolmulation}

 Thanks to the transitivity of the group of automorphisms of the ball while considering the $3$-point Pick problem in the ball we may,
 without loss of generality, restrict ourselves to the following problem
\begin{equation}\nonumber
\BB_n\to\DD,\quad 0\mapsto 0,\; z\mapsto\sigma,\; w\mapsto\tau,
\end{equation}
 where $z\neq w$, $z\neq 0$, $w\neq 0$ and $(\sigma,\tau)\neq(0,0)$.

Let us come back to the formulation of our main problem. In order to do it we repeat a reasoning from \cite{Kos 2015} which allows us to formulate the problems in a more handy way.

Let us denote
$\mathcal D_n:= \{(z, w)\in\BB_n\times\BB_n:z\neq w,z\neq 0, w\neq 0\}$.

A standard Montel-type argument shows that for any $(z, w)\in\mathcal D_n$ and any
$(\sigma,\tau)\in\CC^2\setminus\{(0,0)\}$  there is exactly one $t = t_{z,w,(\sigma,\tau)}>0$ such that the problem
\begin{equation}\label{NP-problem-ball}
\BB_n\to\DD,\quad 0\mapsto 0,\; z\mapsto t\sigma \text{ and } w\mapsto t\tau
\end{equation}
is extremal. It is simple to see that
the mapping
\begin{equation}\nonumber
\mathcal D_n\times(\CC^2\setminus\{(0,0)\})\owns (z, w,\sigma,\tau)\mapsto t_{z,w,(\sigma,\tau)}\in\RR_{>0}
\end{equation}
is continuous.
Moreover, for fixed nodes $z$, $w$ the mapping
$[\sigma:\tau]\mapsto (t_{z,w,(\sigma,\tau)}\sigma, t_{z,w,(\sigma,\tau)}\tau)$
gives a $1-1$ correspondence between the projective plane $\mathbb P^1$ and the set of target
points for which (\ref{NP-problem-ball}) is extremally solvable modulo a unimodular constant.
Saying about the extremal $3$-point Pick problem corresponding to the data $(z, w, [\sigma :\tau])$ we mean the problem
\begin{equation}\nonumber
\BB_n\to\DD,\quad 0\to 0,\; z\mapsto t_{z,w,(\sigma,\tau)}\sigma,\; w\mapsto t_{z,w,(\sigma,\tau)}\tau.
\end{equation}

In particular, the target points $t_{z,w,(\sigma,\tau)}\sigma$ and $t_{z,w,(\sigma,\tau)}\tau$ are determined up to a
unimodular constant.

\subsection{Degenerate and non-degenerate cases}
An extremal $3$-point Pick problem in $D$
\begin{equation}\nonumber
D\to\DD,\quad z_j\mapsto\lambda_j,\; j=1,2,3,
\end{equation}
is called \textit{non-degenerate} if no $2$-point subproblem is extremal.

We divide the set $\mathcal D_n\times\mathbb P^1$ into three sets $\mathcal A$, $\mathcal B$ and $\mathcal C$ with $\mathcal A\cap\mathcal B=\emptyset$, $\mathcal A\cup\mathcal B=\mathcal D_n$ and $\mathcal A$ open.

We say that an element $(z,w,[\sigma:\tau])\in \mathcal D_n\times \mathbb P^1$ belongs to the set $\mathcal A$ (respectively, $\mathcal B$) if and only if its corresponding extremal
$3$-point Pick problem is non-degenerate (respectively, degenerate). Moreover, we define $\mathcal C$ to consist of points $(z,w,[\sigma:\tau])$ such that $0,$ $z$ and $w$ lie in the range of a $2$-extremal. In other words this means that points $0$, $z$ and $w$ are co-linear. It is clear that $\mathcal C$ is a proper analytic set.

Our aim will be the effective description of $\mathcal B$ from which we shall conclude that 
$\mathcal A\setminus \mathcal C$ is connected and thus so is $\mathcal A$. We shall see that for any extremal function $F\in\OO(\BB_n,\DD)$ corresponding to the extremal $3$-point Pick problem $(w,z,[\sigma:\tau])$ there will be an $f:\DD\to\BB_n$ such that $F$ is a left inverse to $f$ (and thus $f$ is a complex $3$-geodesic). 
In the case $(z,w,[\sigma:\tau])\in\mathcal B$ we may effectively pick the extremal function $F$ from a given class of functions (to be defined later) whereas the existence of extremal functions $f$ corresponding to the extremals $F$ for points from $\mathcal A$ will follow from some topological argument (relying on connectivity of $\mathcal A\setminus\mathcal C$) - this idea is the same as in the recent paper \cite{Kos 2015} concerning the same problem but in the polydisc. In the latter case we shall be able to find the class of extremal functions, too.

\subsection{Description of the degenerate case}

The description of the set $\mathcal B$ is given below. Note that the degeneracy of the extremal $3$-point Pick problem
$$\mathbb B_n\to\DD,\; w\mapsto F(w),\; z\mapsto F(z),\; u\mapsto F(u),$$ where $F\in \mathcal O(\BB_n,\DD)$, means that for two distinct points, say $w,z$, we have $c_{\BB_n}^*(w,z)=c_{\DD}^*(F(w),F(z))$. Composing nodes with automorphisms of $\BB_n$ we lose no generality assuming that $w=0$, $F(0)=0$
and $F(z_1,0^{\prime})=z_1$ for some, and consequently applying the Schwarz lemma, for any $z_1\in\DD\setminus\{0\}$. Therefore, the description of $\mathcal B$ reduces to the description of the possible values of $F(u)$ for fixed $u\in\BB_n$ where $F\in\mathcal O(\mathbb B_n,\DD)$ satisfies $F(z_1,0)=z_1$, $z_1\in\DD$.

In other words the problem of description of $\mathcal B$ reduces to the description, for the fixed $w\in\BB_n$, 
of the following set
\begin{equation}\nonumber
\mathcal B(w):=\{F(w):F\in\mathcal O(\BB_n,\DD),F(z_1,0^{\prime})=z_1,z_1\in\DD\}.
\end{equation}
One more reduction is possible, due to the form of the group of automorphisms of $\BB_n$, the point $w$ may be assumed to be from $\BB_2\times\{0\}$ - it is therefore sufficient to discuss the case $n=2$.

Summarizing the set $\mathcal B(w)$ comprises the possible values $\sigma$ of the degenerate extremal $3$-point Pick problem in $\BB_2$ 
\begin{equation}\nonumber
\BB_2\to\DD,\quad (0,0)\mapsto 0,\; (z_1,0)\mapsto z_1,\; w\mapsto\sigma
\end{equation}
for some (any) $z_1\in\DD_*$.

Note that having given a function $G\in\OO(\BB_2,\DD)$ such that $G(z_1,0^{\prime})=z_1$, $z_1\in\DD$  any function of the form
\begin{equation}\nonumber
z\mapsto \overline{\tau}G(\tau z_1,\omega z_2)
\end{equation} 
where $|\tau|=1$, $|\omega|\leq 1$ maps $\BB_2$ to $\DD$ and points $(z_1,0)$ to $z_1$.

Define 
\begin{equation}\nonumber
F_{1,1}(z):=\frac{2z_1(1-z_1)-z_2^2}{2(1-z_1)-z_2^2},\;z\in\BB_2.
\end{equation}
Note that $F_{1,1}\in\OO(\BB_2,\DD)$ - it is a straightforward consequence of elementary transformations. It is also worth mentioning that we deduced the formula for the function $F_{1,1}$ as a function playing a special role in describing the set $\mathcal B(w)$ using the fact that it lies in the closure of the set of non-degenerate points. We also point out that the functions 

For $|\tau|=1$, $|\omega|\leq1$ define
\begin{equation}\label{class F11}
F_{\tau,\omega}(z):=\overline{\tau}F_{1,1}(\tau z_1,\omega z_2)=\frac{2z_1(1-\tau z_1)-\overline{\tau}\omega^2z_2^2}{2(1-\tau z_1)-\omega^2 z_2^2}.
\end{equation} 
Then $F_{\tau,\omega}\in\OO(\BB_2,\DD)$ and $F_{\tau,\omega}(z_1,0)=z_1$. Let us draw Reader's attention to the fact that the functions just introduced are examples showing non-uniqueness of left inverses to complex geodesics in the ball much simpler than the one studied in \cite{Kos-Zwo 2016}. 

The fact that $\mathcal B(w)$ has non-empty interior is one of the main results of \cite{Ama}. Our next theorem precisely describes this set.

\begin{theorem}\label{theorem-degenerate} Let $w\in\BB_2$. Then
\begin{multline}
\mathcal B (w)=m_{w_1}\left(\mathcal B\left(0,\frac{w_2}{\sqrt{1-|w_1|^2}}\right)\right)=\\
m_{w_1}
\left(\overline{\triangle}\left(0,\frac{|w_2|^2}{2-2|w_1|^2-|w_2|^2}\right)\right).
\end{multline}
In particular, the set $\mathcal B(w)$ is a closed Euclidean disc.

Moreover, the extremal $3$-point Pick interpolating functions in the degenerate case may be chosen from a nice class of domains. More precisely,
\begin{equation}\nonumber
\mathcal B(w)=\{F_{\tau,\omega}(w):|\tau|=1,|\omega|\leq 1\}.
\end{equation}
\end{theorem}
In Theorem~\ref{theorem-degenerate} not only do we have the effective description of the set $\mathcal B(w)$ (and thus $\mathcal B$) but also we find the class of extremal mappings which deliver all the possible values in the degenerate case.

In any case it will follows from the above theorem that the set $\mathcal A$ is connected (see Lemma ~\ref{lem: connect}). This fact will be crucial in the proof of the next theorem.

\subsection{Description of the non-degenerate case}\label{subsection-non-degenerate}

Let $\uU_n(\CC)$ denote the $n^2$ dimensional Lie group of unitary matrices. 

Below we present a construction that allows us to derive Theorem~\ref{theorem-main} in the non-degenerate case. As already mentioned, it is sufficient to express it for $n=2$.

Let us consider a mapping $\Phi$ defined on the set 
\begin{equation}\nonumber
\Omega:=\{(x,y,a,U,c)\in\DD_* \times \DD_* \times (0,1)\times \uU_2(\CC) \times (-1,1):x\neq y\}
\end{equation}
by the formula
\begin{equation}\label{def: Phi}
\Phi(x,y,a,U,c):=(\phi_{a,U,c}(x), \phi_{a,U,c}(y), [x m_\gamma(x): y m_\gamma (y)]),
\end{equation}
where $\gamma=\frac{2c}{1+|c|^2}$, $\phi_{a,U,c}(\la):=\chi_w (U (a m_c(\la), (1-a^2)^{1/2} m_c^2(\la)))$, $\lambda\in\DD$. Recall that 
$m_c$ is the idempotent M\"obius map switching $0$ and $c$, and $\chi_w$, where $w=U((1-a^2)^{1/2} c, a c^2)$, is an idempotent automorphism of the Euclidean unit ball switching $0$ and $w$. 

Let us formulate a result which is contained in the proof of Theorem~5.8 in \cite{War 2015}.

\begin{proposition} (see \cite{War 2015})\label{proposition-left-inverse}
Let $a,U,c,\gamma$ be as above. Then there is a holomorphic mapping $F:\BB_2\to\DD$  such that
\begin{equation}\nonumber
F(\phi_{a,U,c}(\lambda))=\lambda m_{\gamma}(\lambda),\;\lambda\in\DD.
\end{equation}
In particular, $\phi_{a,U,c}$ is a $3$-complex geodesic (that is not a $2$-extremal) and $F$ is its left inverse.
\end{proposition}
  
The above proposition implies that the image of $\Phi$ lies in the set $\mathcal A\setminus\mathcal C$ (see Lemma~\ref{lem: PhisubsA} for details). Our aim is to show more. Namely, we have

\begin{theorem}\label{theorem-non-degenerate} With the notation as above,
$\Phi(\Omega)=\mathcal A\setminus \mathcal C$.
\end{theorem}

In other words, there is a correspondence (in fact $2:1$) between the three extremals and the non-degenerate $3$-point Pick problem in the ball. Moreover, as we shall see the solution of the non-degenerate extremal $3$-point Pick problem in the ball may also be taken from a relatively simple class of functions.

Analogous result holds in the polydisc (see \cite{Kos 2015}). However, there are some differences between the case of the ball and that of the polydisc. First note that unlike in the case of the polydisc the case of the unit ball may be easily reduced (due to the form of holomorphic automorphisms) to the two dimensional case. For instance the proof in the polydisc was different in dimension, $2$, $3$ and at least $4$. Secondly the situation in the ball differs from that in the polydisc since the degenerate case in the ball is a big one.

\section{Degenerate case - proof of Theorem~\ref{theorem-degenerate}}
At first we clarify the situation when the degeneracy is `strong'. Roughly speaking the result below says that if two of three $2$-point subproblems corresponding to $(z,w,\xi)$ are extremal, then $(z,w,\xi)$ lies in $\mathcal C$. More precisely:

\begin{lemma}\label{lem: colinear}
Assume that  two of three $2$-point subproblems of the $3$-problem Pick problem 
\begin{equation}\nonumber
\BB_n\to \DD,\quad z_i\mapsto \sigma_i,\; i=1,2,3,
\end{equation}
are extremal.
Then $z_1,z_2, z_3\in \BB_n$ lie on a common complex affine line.
\end{lemma}
\begin{proof} Let $F$ be a function interpolating the above $3$-point Pick problem.
Since automorphisms of the Euclidean ball map complex affine lines into complex affine lines we may assume that $z_3=0$, $z_1=(z_1',0)$ and that the subproblems comprising $z_1,z_3$ and $z_2, z_3$ are extremal. We may also assume that $\sigma_3=0$ and $\sigma_1= z_1$. Suppose the claim does not hold so $z_2=\lambda_0 v$ for some $\lambda_0\in\DD\setminus\{0\}$, where $||v||=1$ and $|v_1|<1$. Then it is clear that $F(\la,0) = \la$, and $F(\lambda v)=e^{i\theta}\la$ for some $\theta\in\RR$ and all $\la\in \DD$. 

Consequently,
\begin{equation}\nonumber
(c_{\BB_n}^*((z_1^{\prime},0),\la v))^2\geq c_{\DD}^*(F(z_1^{\prime},0),F(\la v))^2=(c_{\DD}^*(z_1^{\prime},e^{i\theta}\la))^2,\;\la\in\DD
\end{equation}
which is equivalent to 
\begin{equation}\nonumber
|1-e^{i\theta}\la\overline{z}_1|^2\leq|1-\la v_1\overline{z}_1|^2,\;\la\in\DD.
\end{equation}
Substituting $\la$ in the unit circle such that $\re(e^{i\theta} \la \bar z_1)=-|z_1|$ we get a contradiction.

\end{proof}

Now we proceed to the proof of Theorem~\ref{theorem-degenerate}
\begin{proof}[Proof of Theorem~\ref{theorem-degenerate}] We shall make use of the invariance of the sets $\mathcal B(w)$ under automorphisms so at first we study the case of $w\in\{0\}\times\DD$. To do this define
\begin{equation}\nonumber
u(z_2):=\sup\{|F(0,z_2)|:F\in\mathcal O(\BB_2,\DD), F(z_1,0)=z_1\},\; z_2\in\DD.
\end{equation}

It is simple to see that $u(z_2)=u(|z_2|)$.

We claim that 
\begin{equation}\label{claim}
u(z_2)\leq\frac{|z_2|^2}{2-|z_2|^2},\; z_2\in\DD.
\end{equation}

Fix $1>z_2\geq 0$. Take $F\in\mathcal O(\BB_2,\DD)$ such that $F(z_1,0)=z_1$, $z_1\in\DD$. Without loss of generality assume that $x:=F(0,z_2)\geq 0$. The holomorphic contractibility of the Carath\'eodory distance gives
\begin{equation}\nonumber
(c_{\BB_2}^*((z_1,0),(0,z_2)))^2\geq (c_{\DD}^*(F(z_1,0),F(0,z_2)))^2=c_{\DD}^*(z_1,x)^2
\end{equation}
for any $z_1\in\DD$. The last inequality may be written in the form
\begin{equation}\nonumber
1-(1-|z_1|^2)(1-|z_2|)^2\geq \left|\frac{z_1-x}{1-xz_1}\right|^2,\;z_1\in\DD.
\end{equation}

Consider only $z_1\in(-1,1)$. Then the last inequality is equivalent to
\begin{equation}\nonumber
x^2(1-z_2^2)(1+z_1^2-z_1^2z_2^2)-2xz_1(1-z_1^2)(1-z_2^2)-z_2^2(1-z_1^2)\leq 0,\;z_1\in(-1,1).
\end{equation}


Consequently, 
\begin{equation}\nonumber
x^2(1+z_1^2-z_1^2z_2^2)-2xz_1(1-z_2^2)-z_2^2\leq 0,\; z_1\in(-1,1).
\end{equation}

Passing with $z_1\to-1$ we get that
\begin{equation}\nonumber
x^2(2-z_2^2)+2x(1-z_2^2)-z_2^2\leq 0.
\end{equation}
The last is equivalent to the inequalities $-1\leq x\leq\frac{z_2^2}{2-z_2^2}$, which finishes the claim.

The inequality (\ref{claim}) implies that for any $w_2\in\DD$ we get the inclusion
$\mathcal B(0,w_2)\subset\{\sigma:|\sigma|\leq \frac{|w_2|^2}{2-|w_2|^2}\}$. To see the opposite inclusion fix $w_2\in\DD$. 
Then manipulating with $|\tau|=1$, $|\omega|\leq 1$ we easily find that for any 
$|\sigma|\leq \frac{|w_2|^2}{2-|w_2|^2}$ there is a function $F:=F_{\tau,\omega}$ such that $F(0,w_2)=\sigma$ which gives the desired description of $\mathcal B(0,w_2)$.

\vskip1cm

Recall that $\chi_{(w_1,0)}(z)=\left(m_{w_1}(z_1),\frac{\sqrt{1-|w_1|^2}}{1-\overline{w}_1z_1}z_2\right)$, $z\in\BB_2$.

Note that for any function $F\in\mathcal O(\BB_2,\DD)$ such that $F(z_1,0)=z_1$, $z_1\in\DD$ the function $G:=m_{w_1}\circ F\circ \chi_{(w_1,0)}\in\OO(\BB_2,\DD)$ and it satisfies the equality $G(z_1,0)=z_1$, $z_1\in\DD$, too. 

Since $\chi_{(w_1,0)}(w)=\left(0,\frac{w_2}{\sqrt{1-|w_1|^2}}\right)$ this implies that 
$$
\mathcal B (w)=m_{w_1}\left(\mathcal B\left(0,\frac{w_2}{\sqrt{1-|w_1|^2}}\right)\right),
$$
which gives the desired description. In particular, for any $w\in\BB_2$ the set $\mathcal B(w)$ is the closed Euclidean disc lying in $\DD$ - more precisely it is a closed disc with respect to the Poincar\'e distance (the function $c_{\DD}^*$) centered at $w_1$.

\bigskip

The above procedure also allows us to construct functions that give all the values $F(w)$ from the set $\mathcal A(w)$. It follows from the above reasoning (and the result for $w\in\{0\}\times\DD$) that these extremal values will be attained by functions from the class 
$$
\{m_{w_1}\circ F_{\tau,\omega}\circ \Psi_{(w_1,0)}: |\tau|=1, |\omega|\leq 1\}.
$$

But this class (formally depending on $w$) is the same for all $w$ and consequently it coincides with that for $(0,w_2)$ (some or any). To see this, it is sufficient to make elementary calculations to get that
\begin{equation}\nonumber
m_{w_1}\circ F_{\tau,\omega}\circ \Psi_{(w_1,0)}(z)=\frac{2\left(1-z_1\frac{\overline{w_1}-\tau}{1-\tau w_1}\right)z_1
+\omega^2\frac{\overline{\tau}-w_1}{1-\tau w_1} z_2^2}{2\left(1-z_1 \frac{\overline{w_1}-\tau}{1-\tau w_1}\right)+\omega^2\frac{\overline{\tau}\overline{w}_1-1}{1-\tau w_1} z_2^2}.
\end{equation}
It is easy to observe that the last form is the function $F_{\tilde\tau,\tilde\omega}$, where $\tilde\tau=\frac{\overline{w}_1-\tau}{1-\tau w_1}$, $\tilde\omega^2=-\omega^2\frac{\overline{\tau}\overline{w}_1-1}{1-\tau w_1}$ .
\end{proof}

One may conclude more from the proof of Theorem~\ref{theorem-degenerate}. Note that in contrast with the situation in the polydisc in the case of the ball the set of degenerate extremal $3$-point Pick problem is a `big' one in the sense that the set $\mathcal B$ has non-empty interior. Moreover, as formulated in Theorem~\ref{theorem-degenerate} the class of functions which gives all the possible values in this case is, up to automorphisms of the ball, recovered from a nice class of functions:
\begin{equation}\label{class FD}
\mathcal F_D:=\{F_{\tau,\omega}:|\tau|=1,|\omega|\leq 1\},
 \end{equation} 
where $F_{\tau,\omega}$ is given by \eqref{class F11}.
It is interesting to note that in order to parametrize the boundary $\partial\mathcal B$ the corresponding class of functions may be chosen from the above one with the additional assumption that $|\omega|=1$.

\section{Non-degenerate case - proof of Theorem~\ref{theorem-non-degenerate}}

\subsection{3-complex geodesics}

It was proven in \cite{War 2015} (Theorem 5.8) that any $3$-extremal in the unit ball is a complex $3$-geodesic. We recall the reasoning that led to this result.

Let $f:\DD\to \BB_2$ be a $3$-extremal in the unit ball which is not a $2$-extremal. Recall that then $f$ is, up to a composition with an automorphism of $\BB_n$, of the form $$\lambda \mapsto (a \lambda, \sqrt{1-a^2} \lambda m_\alpha (\lambda)),$$ where $a\in [0,1)$ and $\alpha \in \DD$ (see \cite{Kos-Zwo 2015}, Section 3).

It was also proved in \cite{Kos-Zwo 2015} that any such $f$ with $\alpha=0$ admits a left inverse of the form 
$$F(z) = \frac{z_1^2}{2- a^2} + \frac{2 \sqrt{1-a^2} z_2 }{2-a^2}$$
thus showing that such an $f$ is a complex $3$-geodesic.

Later, Warszawski in \cite{War 2015} (Theorem 5.8) showed additionally that any $3$-extremal $f$ that is not a $2$-extremal is actually equivalent with a $3$-complex geodesic of the form $$\lambda \mapsto (a m_c(\lambda), \sqrt{1-a^2} m_c^2(\lambda)),\; a\in[0,1).$$ A more detailed result on $3$-geodesity of $3$-extremals (not being $2$-extremals) proven in \cite{War 2015} is presented in Proposition~\ref{proposition-left-inverse}.

Remark that the above facts show that one may choose a relatively small class (modulo automorphisms of the unit ball) of left inverses that would be good for  all $3$-extremals. 
This class of functions equals (up to automorphisms)
\begin{equation}\label{class FND}
\mathcal F_{ND}:=\left\{\frac{z_1^2}{2-a^2}+\frac{2\sqrt{1-a^2}z_2}{2-a^2}:a\in [0,1)\right\}.
\end{equation}
These observations will be crucial in our subsequent considerations.

\subsection{Left inverses}
As we have already mentioned the function $F$ given by the formula 
$F(z) = \alpha z_1^2 + \beta z_2$ with $\alpha=\frac{1}{2-a^2}$, $\beta=\frac{2\sqrt{1-a^2}}{2-a^2}$ is a left inverse to $\lambda \mapsto (a\lambda, b \lambda^2)$, where $a,b\geq 0$, $a^2 + b^2 =1$. Elementary calculations also show that $|F|$ attains its maximum on the topological boundary of $\partial \BB_2$ along the algebraic set $\{z\in\partial\BB_2: b z_1 = a^2 z_2\}$. 

This brings us to the following general situation: let $D$ be a smooth domain in $\CC^n$ and $F\in \mathcal O(\overline D)$ be such that $F(D)\subset \DD$. Assume that there exists $g\in \mathcal O(\overline D)$ such that $\{z\in \partial D: |F(z)|=1\} = \{z\in \partial D:\ g(z)=0\}\cap \partial D$. Some methods coming from analytic geometry will allow us to show that there is $\alpha\in \NN$ such that $$F_K(z):=\frac{F(z)}{\sqrt {1 - K g^\alpha (z)}}\in \DD,\quad z\in D,$$ whenever $K>0$ is small enough.

Actually, let $Z:=\{|F|=1\}\cap \partial D$. \L ojasiewicz inequality (see e. g. \cite{Loj 1991})  gives 
\begin{equation}\label{eqgen} C\dist(z, Z)^\alpha\leq 1-|F(z)|^2
\end{equation}
for $z\in \partial D$.
On the other hand, since $Z$ is the zero set of $g$ on $\partial D$, for any $x\in Z$ we get that $|g(z)|=|g(z)-g(x)|\leq C' |z-x|$, so the following trivial inequality holds:
\begin{equation}\label{eqgen1} |g(z)|\leq C' \dist(z, Z).
\end{equation}
Now \eqref{eqgen} and \eqref{eqgen1} together give $$K|g(z)|^\alpha + |F(z)|^2 <1,\quad z\in \partial D,$$ where $K$ is small enough. The maximum principle finishes the proof of our claim.

Below we shall provide the Reader with an elementary proof of the above fact for $D$ being the unit ball 
and $F$ given as above by the formula $F(z) = \alpha z_1^2 + \beta_2 z_2$ and the function $g$ given by the formula 
$g(z) = b z_1 - a^2 z_2$. Beyond being elementary it has other advantages as it gives explicit formulas for $\alpha$ and $K$ satisfying the claim.

Let us fix the situation that we study. Let 
\begin{equation}\nonumber
a\in(0,1),\;b:=\sqrt{1-a^2},\;
F(z):=\frac{1}{2-a^2}(z_1^2+2bz_2).
\end{equation}
As already mentioned $F$ maps $\BB_2$ into $\DD$ and 
\begin{equation}\label{eq:left-inverse}
F(a\la,b\la^2)=\lambda^2,\;\lambda\in\DD.
\end{equation}

\begin{lemma}\label{lem:nonunique} With the notation as above for any $0\leq\varepsilon<1$ the function
\begin{equation}\nonumber
F_{\varepsilon}(z):=
 \frac{F(z)}{\sqrt{1 -\frac{\varepsilon^2}{(2-a^2)^2}(bz_1^2 - a^2 z_2)^2}}
 \end{equation}
  maps $\BB^2$ into $\DD$.
\end{lemma}
\begin{proof}
To show the above property it is sufficient to show that
\begin{equation}\nonumber
|z_1^2+2bz_2|^2+\varepsilon^2|a^2z_2-bz_1^2|^2<(2-a^2)^2
\end{equation}
on $\BB^2$.

For $\varepsilon>0$ such that $2b\geq \varepsilon^2a^2b$ it is sufficient to show that
\begin{equation}\nonumber
f(x)\leq f(b)=(2-a^2)^2,\;x\in[0,1],
\end{equation}
where $f(x):=(1-x^2+2bx)^2+\varepsilon^2(a^2x-b(1-x^2))^2$.

Since $f$ is a polynomial of degree four with the positive leading coefficient, $f(0)=1+\varepsilon^2b^2$, $f(1)=4b^2+\varepsilon^2(1-b^2)^2$, $f(b)=(1+b^2)^2$, $f^{\prime}(b)=0$, $f^{\prime\prime}(b)=-4(2-a^2)+2\varepsilon^2(2-a^2)^2$ we easily conclude the desired inequality for $\varepsilon$ such that $f^{\prime\prime}(b)<0$ - and this is the case for $0<\varepsilon<1$.
\end{proof}


In the sequel we shall need the following  observation:
\begin{lemma}\label{lem: PhisubsA} Let us keep the notation as in Subsection~\ref{subsection-non-degenerate}.
Let $x,y\in \DD\setminus\{0\}$ be such that $x\neq y$. Then the problem
$$\BB_2\to\DD,\quad 0\mapsto 0,\;\varphi(x)\mapsto x m_\gamma(x),\;\varphi(y) \mapsto  y m_\gamma(y),
$$
where $\varphi = \varphi_{a,U,c}$, is extremal, non-degenerate and omits $\mathcal C$. In other words $\Phi(x,y,a,U,c)\in\mathcal A\setminus\mathcal C$ for any $(x,y,a,U,c)\in\Omega$.
\end{lemma}
\begin{proof}
Extremality is clear, as Blaschke products are $3$-extremals in $\DD$.

We shall show that the subproblem 
\begin{equation}\nonumber
\BB_2\to\DD,\; 0\mapsto 0,\; \varphi(x) \mapsto x m_\gamma (x)
\end{equation} 
is not extremal. The proof for two other $2$-point subproblems is based on the same idea. 

Note that extremality of the above $2$-point subproblem is equivalent to the equality (use (\ref{eq:left-inverse} 
and the definition of the $\varphi$.
$$c_{\BB_2}^* ((ac^2,(1-a^2)^{1/2} c),(a m_c^2(x),(1-a^2)^{1/2} m_c(x)))=c_{\DD}^*(c^2, m_c^2(x)).$$ 
Put $y=m_c(x)$ and $b=(1-a^2)^{1/2}$. The last equality means that there is a geodesic $\psi$ in $\BB_2$ such that $\psi(c^2)=( ac^2,bc)$ and $\psi(y^2) = (ay^2,by)$. Let $F_{\varepsilon}$ be functions constructed in Lemma~\ref{lem:nonunique} such that $F_\varepsilon(b \la, a\la^2)=\la^2$, $0\leq \varepsilon<1$. Note that $F_\varepsilon\circ \psi$ is the identity for any $0\leq \varepsilon<1$, so due to the form of $F_{\varepsilon}$ we get the equality $a\psi_1^2 \equiv b^2 \psi_2$ on $\DD$. This means that the image of $\psi$ lies in the variety $\{ (\la, a/b^2  \la^2), \la \in \DD\}$. Since geodesics lie on affine lines we find that $a=0$. But then $c_{\BB_2}^*((0,c), (0,y)) = c_{\DD}^*(c^2, y^2),$ a contradiction.

Finally, we shall show that $0$, $\phi(x)$ and $\phi(y)$ are not colinear. Since automorphisms of $\BB_2$ map affine lines into affine lines the assertion is equivalent to an obvious fact that $(a m_c(0), bm_c^2(0))$, $(a m_c(x), b m_c^2(x))$ and $(a m_c(y), bm_c^2(y))$ do not lie on an affine line.
\end{proof}

\subsection{Openness of the range of $\Phi$}

\begin{lemma}\label{lem:open} The continuous mapping $\Phi:\Omega\to\mathcal A\setminus\mathcal C$ is two-to-one. Moreover, $\Phi(x,y,a,U, c) = \Phi(-x, -y, a, U, -c)$. 

In particular, $\Phi$ is locally injective on any domain composing of points such that $x\neq y$.
\end{lemma}
\begin{proof} It follows from the definition that the equality $\Phi(x,y,a,U,c)=\Phi(-x,-y,a,U,-c)$ holds for any $(x,y,a,U,c)\in\Omega$. What remains to be proven is to show that if 
$\Phi(x_1, y_1, a_1, U_1, c_1) = \Phi(x_2,y_2, a_2, U_2, c_2)$ then either $(x_1, y_1, a_1, U_1, c_1) = (x_2,y_2, a_2, U_2, c_2)$ or
$(x_2, y_2, a_2, U_2, c_2) = (-x_1,-y_, a_1, U_1, -c_1)$. 

Assume that $\Phi(x_1, y_1, a_1, U_1, c_1) = \Phi(x_2,y_2, a_2, U_2, c_2)$ .To simplify the notation let us denote $\phi_i=\phi_{a_i, U_i, c_i}$ and $\gamma_i = \frac{2c_i}{1+|c_i|^2}$, $b_i = (1-a^2_i)^{1/2}$, $w_i = U_i(a_i c_i, b_i c_i^2)$, $\chi_i = \chi_{w_i}$, $i=1,2.$
Notice that if the problem
\begin{equation}\nonumber
\BB_2\to\DD,\quad 
0\mapsto 0,\;
\phi_1(x_1)\mapsto x_1 m_{\gamma_1}(x_1),\;
\phi_1(y_1) \mapsto y_1 m_{\gamma_1}(y_1)
\end{equation} 
is interpolated by a function $F$, then the problem
\begin{equation}\nonumber
\BB_2\to\DD,\quad
0\mapsto 0,\;\phi_2(x_2)\mapsto x_2 m_{\gamma_2}(x_2),\;
\phi_2(y_2) \mapsto y_2 m_{\gamma_2}(y_2)
\end{equation}
is interpolated by a function $\omega F$ for some unimodular $\omega$. 
Applying this observation to $F:=m_{c_1^2} \circ F_\varepsilon \circ U_1^{{-1}}\circ \chi_{1}$, where $F_\varepsilon$ is a left inverse to $\la \mapsto (a_1 \lambda, (1-a_1^2)^{1/2}  \lambda ^2)$ constructed in Lemma~\ref{lem:nonunique}, $0\leq\varepsilon<1$, we get that there is some $\omega\in\partial\DD$ such that
\begin{equation}\label{eq:Fphi} F_\varepsilon \circ U_1^{-1} \circ \chi_1 \circ \phi_2 (\la) = m_{c_1^2}  (\omega \la m_{\gamma_2}(\la)), \quad \la\in \DD,
\end{equation} and any $\varepsilon<\sqrt 2$.
\begin{claim}\label{claim:identity}
Let $\psi\in\mathcal O(\DD,\BB_2)$. Assume that there is a Blaschke product of degree $2$ such that $$F_\varepsilon \circ \psi \equiv B$$ for any $0\leq \varepsilon<1$. Then $B\equiv\eta m^2$ for some M\"obius map $m$ and unimodular constant $\eta$. Moreover, $\psi_2 \equiv b_1 m^2$ and $\psi_1^2 \equiv a_1^2 m^2$.
\end{claim}

\begin{proof}[Proof of the claim] Since the equality $F_\varepsilon \circ \psi(\la) = B(\la)$ for $\la\in\DD$ holds for all $\eps$ small enough we get from the form of $F_{\eps}$ that $b_1 \psi_1^2 (\la) = a_1^2 \psi_2(\la)$ for $\la\in \DD$. Direct computation gives equalities 
\begin{equation}\nonumber
B(\lambda)=F_\varepsilon(\psi(\la))= F(\psi(\lambda))= \psi_1^2(\la) / a_1^2,\;\la \in \DD
\end{equation} 
and the claim follows.
\end{proof}

We come back to the proof of the lemma. Making use of the claim we see that 
\begin{equation}\label{eq: mcmalpha} m_{c_1^2} (\omega \lambda m_{\gamma_2}(\lambda)) = \eta m^2_{\alpha}(\lambda),\quad \la\in\overline{\DD},
\end{equation}
for some idempotent M\"obius map $m_\alpha$ and unimodular constant $\eta$. Differentiating \eqref{eq: mcmalpha} at the point $\alpha$ we find that $\gamma_2 = \frac{2 \alpha}{1+ |\alpha|^2}$; in particular, $\alpha\in \RR$. Putting $\lambda =0$ to \eqref{eq: mcmalpha} one gets  $c_1^2 = \eta \alpha^2$ and, therefore, $\eta=1$ and $c_1^2 = \alpha^2$. Consequently, $\gamma_2 = \pm \frac{2c_1}{1+|c_1|^2}$, which means that $c_2=\pm c_1.$ Putting $\la =1$ to \eqref{eq: mcmalpha} one gets $\omega=1$. 

Replacing, if necessary, $(x_2, y_2, c_2)$ with $(-x_2, -y_2, -c_2)$ we may assume that $c_2=c_1$. What we have to do now it so show that $x_1=x_2$, $y_1=y_2$, $U_1=U_2$ and $a_1=a_2$.

Note that Claim~\ref{claim:identity} implies that 
\begin{equation}\nonumber
U_1^{-1}\circ \chi_1 \circ \phi_2(\la)=\mathcal J\circ U_1^{-1}\circ \chi_1 \circ \phi_1(\la),\quad \la\in\DD,
\end{equation}
where $\mathcal J$ is either the identity matrix or $\mathcal J = \diag(-1, 1)$. Substituting $\la=0$ in the above equality we infer that $\mathcal J$ fixes $(a_1 c_1, b_1 c_1^2)$, whence one deduces that $\mathcal J$ is the identity. Consequently, $\phi_2 \equiv \phi_1$ so $$\chi_2 \circ U_2 (a_2 \la, b_2 \la^2) = \chi_1 \circ U_1 (a_1 \la, b_1 \la^2),\quad \la \in \DD.$$ Thus $\chi_2(0)=\chi_1(0)$, so $\chi_2 \equiv \chi_1$. Since $U_1$ and $U_2$ are isometries, the above relations imply that $||(a_1, b_1 \la)|| = ||(a_2, b_2 \la)||$ for any $\la \in \DD$. Remembering that $a_i,b_i$ are positive, we simply deduce that $a_1=a_2$, so $U_1=U_2$ trivially. Finally, the equalities $x_1=x_2$ and $y_1 = y_2$ follow immediately from the relations $\phi_1(x_1) = \phi_2 (x_2)$ and $\phi_1(y_1) = \phi_2(y_2)$.
\end{proof}


\subsection{Closedness of the range of $\Phi$}

\begin{lemma}\label{lemma:closed}
The range $\Phi(\Omega)$ is closed in $\mathcal A \setminus \mathcal C$.
\end{lemma}

\begin{proof} We are using the notation as in Proposition~\ref{proposition-left-inverse}. 
Let a sequence
 $((x_n, y_n, a_n, U_n, c_n))$ in $\Omega$, convergent to $(x_0,y_0, a_0, U_0, c_0)$ in $\bar \Omega$, be such that the sequence $(\Phi(x_n, y_n , a_n, U_n, c_n))=:((x_n,y_n,\xi_n))$ is convergent to $(z_0, w_0, \xi_0)\in \mathcal A \setminus \mathcal C.$ Put $\sigma_n:= x_n m_{\gamma_n} (x_n)$ and $\tau_n:= y_n m_{\gamma_n}(y_n)$. Assume additionally that sequences $(\sigma_n)$ and $(\tau_n)$ are convergent to $\sigma_0$ and $\tau_0$. We easily get that $\sigma_0,\tau_0\in\DD$. Additionally, it follows from the continuity argument that the problem
\begin{equation}\nonumber
\BB_2\to\DD,\quad 0\mapsto 0,\; z_0\mapsto \sigma_0,\; w_0\mapsto \tau_0
\end{equation}
is extremal.

And once more since the sequence $(\Phi(x_n,y_n,a_n,U_n,c_n))$ converges to an element from $\mathcal A\setminus\mathcal C$ we get that $a_0\in(0,1)$.

Note that if we show that $c_0\in(-1,1)$ then we are done.

Therefore, we may assume that $c_n$ converges to $1$ (the case $c_0=-1$ follows from fact that $\Phi$ is even with respect to the variables $x,y,c$). In such a case the sequence $(\gamma_n)$ tends to $\gamma_0=1$. We shall consider three cases. 

The first one is when $x_0=y_0=1$. Then 
\begin{multline}\nonumber c_{\DD}^*(\sigma_0, \tau_0)= \lim c_{\DD}^*(m_{\gamma_n}(x_n), m_{\gamma_n}(y_n)) =
\\ \lim  c_{\DD}^* (x_n, y_n) \geq\lim  c_{\BB_2}^*(z_n, w_n) = c_{\BB_2}^*(z_0, w_0).
\end{multline}
On the other hand we have the trivial inequality $c_{\DD}^*(\sigma_0, \tau_0) \leq c_{\BB_2}^*(z_0,w_0)$. Both these relations imply that the the $2$-point problem 
\begin{equation}\nonumber
\BB_2\to\DD,\; z_0\mapsto \sigma_0,\; w_0\mapsto \tau_0
\end{equation}
is extremal, so $(z_0, w_0, \xi_0)\in \mathcal B$; a contradiction.

Now suppose that $x_0\in \DD$. Then $\sigma_0 = x_0$, which implies that the $2$-point problem $0\mapsto 0$, $z_0\mapsto \sigma_0$ is extremal - contradiction, either. If $y_0\in \DD$ we proceed similarly.

We are left with the case when $x_0,y_0\in\partial\DD$ and either $x_0\neq 1$ or $y_0\neq 1$. But this means that either $\tau_0$ or $\sigma_0$ does not lie in the unit disc, which leads to a contradiction, either.
\end{proof}


\subsection{Connectedness of non-degenerate set}

\begin{lemma}\label{lem: connectfiber}
Suppose that $z,w\in \BB_2$ do not lie on a line passing through $0$. Then the set of points $\xi=[\sigma:\tau]$ such that $(z,w,\xi)\in \mathcal A$ is connected.
\end{lemma}

\begin{proof}
To prove the assertion we shall show that points $\xi$ for which $(z,w,\xi)\in \mathcal B$ form a union of three disjoint simply-connected and closed sets.

Two of them are closed discs consisting of points $\xi=[\sigma:\tau]$ such that one of the following $2$-point interpolation subproblems $$0\mapsto 0,\; z\mapsto t_{z,w,(\sigma, \tau)}\sigma\quad \text{ or }\quad  0\mapsto 0,\; w\mapsto t_{z,w,(\sigma, \tau)}\tau$$ is extremal (use Theorem~\ref{theorem-degenerate}).

To get the assertion it thus suffices to show that the set of $\xi$ such that $$z\mapsto t_{z,w,(\sigma, \tau)}\sigma,\; w\mapsto t_{z,w,(\sigma, \tau)}\tau$$ is extremal is simply connected. The reason for this is that all the sets constructed here are disjoint, according to Lemma~\ref{lem: colinear}.

Let us compose nodes and points with proper automorphisms so that we are in a position that allows us to apply Theorem~\ref{theorem-degenerate} again. Then one can see that we aim at describing $[\sigma: \tau]$ such that $m_\sigma (\tau)= ||x||=x_1>0$ (and thus $x_2=0$) and $c_{\DD}^*(\tau,z_1)\leq\frac{|z_2|^2}{2-2|z_1|^2-|z_2|^2}=:r$, where $x = (x_1, x_2)= (x_1,0)= U(\chi_z (w))$ with suitably chosen unitary matrix $U$.

But then $\tau=\frac{\sigma-x_1}{1-x_1\bar\sigma}$. What we need is to show that the set
\begin{equation}\nonumber
\left\{\frac{\sigma-x_1}{\sigma-x_1\bar\sigma\sigma}:\sigma\in\DD,\;c_{\DD}^*(\sigma,z_1)\leq r\right\}\subset\hat{\mathbb C}
\end{equation}
is simply connected
(if $\sigma=0$ the fraction with $0$ in the denominator is understood to be $\infty$).

We shall prove it. Since 
\begin{equation}\nonumber
\frac{\sigma-x_1}{\sigma-x_1\bar\sigma\sigma}=1-x_1\frac{1-|\sigma|^2}{\sigma-x_1|\sigma|^2}
\end{equation}
we easily reformulate the problem.

Consider the function
\begin{equation}\nonumber
\psi:\DD\owns\sigma\mapsto\frac{\sigma-x_1|\sigma|^2}{1-|\sigma|^2}\in\mathbb C.
\end{equation}
Since any closed Poincar\'e disc in $\DD$ is is the Euclidean disc to finish the proof of the lemma it is sufficient to show that for any closed disc $K\subset\DD$ the set $\psi(K)$ is simply connected.

To prove it note that the (real) Jacobian of $\psi$ equals 
\begin{equation}\nonumber
\frac{1+|\sigma|^2-2\operatorname{Re}(x_1\overline{\sigma})}{(1-|\sigma|^2)^3}>0,\;\sigma\in\DD.
\end{equation}
Consequently $\psi$ is a local diffeomorphism. Note also that $\psi$ is proper onto its image (equal to $\mathbb C$). Consequently, $\psi$ is a finite topological covering, which easily implies the desired conclusion.
\end{proof}

\begin{lemma}\label{lem: connect}
The set of non-degenerate points $(z,w,\xi)$ is connected.
\end{lemma}
\begin{proof}
Note that the set $\mathcal C_1$ of $(z,w)$ lying on the same complex line is analytic, co removing it does not affect the connectedness.

Thus the result is a consequence of the trivial fact that the natural projection $$\mathcal A\setminus \mathcal C  \ni (z,w,\xi)\mapsto (z,w) \in \mathcal D_n \setminus \mathcal C_1$$ is open and Lemma~\ref{lem: connectfiber} which says, that fibers of the above projection are connected.
\end{proof}

\begin{proof}[Proof of Theorem~\ref{theorem-non-degenerate}]
It follows from Lemmas~\ref{lem:open} and \ref{lemma:closed} that the range of the mapping $\Phi$ is open and closed in $\mathcal A \setminus \mathcal C$ which is connected, according to Lemma~\ref{lem: connect}.
\end{proof}

\section{Proof of Theorem~\ref{theorem-main}}
\begin{proof}[Proof of Theorem~\ref{theorem-main}]
The existence of $f$ as in the theorem easily implies the extremality of the $3$-point Pick problem. So assume the problem formulated in the theorem is extremal.

First recall that if the problem
\begin{equation}\label{eq: proofthm}
\BB_n \to \DD,\quad z_j\mapsto \mu_j: = F(z_j),\; j=1,2,3,
\end{equation}
is degenerate, then by Lempert's theorem there is a complex geodesic $f$ in $\BB_n$ passing through (at least) two of the nodes such that $F\circ f$ is a M\"obius map (see Section~\ref{sec: N=2}). Thus the theorem in this case is clear.

Let us assume that \eqref{eq: proofthm} is non-degenerate. As mentioned in Section~\ref{sec: refolmulation} one may assume that $z_3=0$ and $\la_3=0$.

It is clear that the assertion is clear if $(z_1, z_2, [\mu_1 : \mu_2])$ lies in the range of $\Phi,$ which is equal to $\mathcal A\setminus \mathcal C$ (see Theorem~\ref{theorem-non-degenerate}).

Therefore, it suffices to show that the theorem holds if $(z_1, z_2, [\mu_1: \mu_2])$ is in $\mathcal C$. Recall that this means that $z_1$ and $z_2$ lie on a complex line. Let $z_j=\la_j v$, $j=1,2$, for some $v\in\mathbb C^n$, $||v||=1$. Let $U$ be a unitary matrix such that $U(1,0,\ldots,0)=v$. Define $B(\la):=F(U(\la,0,\ldots,0))$ and $f(\la):=U(\la,0,\ldots,0)$, $\la\in\DD$. $f$ interpolates the extremal $3$-point problem
\begin{equation}\nonumber
\DD\to\BB_n,\;0\mapsto 0,\;\la_j\to\la_j v,\quad j=1,2.
\end{equation}
Then the problem
\begin{equation}\nonumber
\DD\to\DD,\quad 0\mapsto 0=B(0),\; \la_j\mapsto B(\la_j),\quad j=1,2,
\end{equation}
is extremal. Actually, otherwise there would be a $g\in\OO(\DD,\DD)$ with $g(0)=0$, $g(\la_j)=B(\la_j)=F(z_j)$, $j=1,2,$ and $g(\DD)$ lying relatively compactly in $\DD$. Then define $G$ with $G(z):=g(\pi(U^{-1}(z)))$, $z\in\BB_n$, where $\pi$ is the projection onto the first variable. Then $G\in \mathcal O(\BB_n,\DD)$, $G(0)=0$, $G(z_j)=F(z_j)$, $j=1,2$ and $G(\BB_n)$ lies relatively compactly in $\DD$, which contradicts the extremality of $F$. Therefore, $B$ is a Blaschke product of degree at most $2$ which completes the proof.
\end{proof}

\section{Relations with the Green function with two poles}
For a domain $D\subset\mathbb C^n$, $p,q,z\in D$, $p\neq q$ define
$l_D(p;q;z)$ as the minimum of three values
\begin{align*}
\inf\{|\lambda\sigma|:\exists{f}\in\OO(\DD,D):f(0)=z,f(\la)=p,f(\sigma)=q\},\\
\inf\{|\la|:\exists{f}\in\OO(\DD,D):f(0)=z,f(\la)=p\},\\
 \inf\{|\la|:\exists{f}\in\OO(\DD,D):f(0)=z,f(\la)=q\}.
\end{align*}
We also define
\begin{equation}\nonumber
c_D(p;q;z):=\sup\{|F(z)|:F\in\OO(D,\DD),F(p)=F(q)=0\}.
 \end{equation} 

Then $c_D\leq l_D$. Let $D$ be additionally bounded and $z\neq p,q$. Then we may find an $F\in\OO(D,\DD)$ with $F(p)=F(q)=0$ and $F(z)=c_D(p,q,z)=:\tau$. Note that the $3$-point Pick problem
 \begin{equation}\nonumber
 D\to\DD,\quad p\mapsto 0,\; q\mapsto 0,\; z\mapsto\tau
 \end{equation}
 is then extremal and $F$ interpolates it. It is then simple that the existence of a holomorphic $f:\DD\to D$ passing through three points $p,q$ and $z$ such that $F\circ f$ is a non-Blaschke product of degree at most two or the existence of a holomorphic $f:\DD\to D$ passing through at least two of the points $p,q,z$ such that $F\circ f$ is a non-constant Blaschke product of degree one (note that in the later case $f$ must necesserily pass through $z$!) implies the equality
\begin{equation}\label{eq:lem-car}
c_D(p;q;z)=l_D(p;q;z).
\end{equation} 
The above equality implies that the Green function with logarithmic poles at $p$ and $q$ at $z$ coincides with both $\log c_D(p;q;z)=\log l_D(p;q;z)$ and thus this shows that in such a case the conjecture of Coman with two poles does hold (see \cite{Com 2000}).
  
In view of Theorem~\ref{theorem-main} the equality (\ref{eq:lem-car}) holds for $D=\BB_n$ and thus this gives another proof showing that the conjecture of Coman holds for the unit ball with two logarithmic poles (see \cite{Com 2000}, \cite{Edi-Zwo 1998}).


\begin{thebibliography}{10}




\bibitem{Agl-Lyk-You 2013} \textsc{J.~Agler, Z.~Lykova, N.~J.~Young}, \textit{Extremal holomorphic maps and the symmetrized bidisc}, Proc. London Math. Soc. (3) 106 (2013) 781-818.


\bibitem{Agl1} {\sc J.~Agler}, {\it Some interpolation theorems of Nevanlinna-Pick type}, preprint, 1988.

\bibitem{Agl2}  {\sc J.~Agler, J.E.~McCarthy}, {\it Nevanlinna-Pick interpolation on the bidisk}, J. Reine Angew. Math. 506 (1999), 191--204. 

\bibitem{Agl3}  {\sc J.~Agler, J.E.~McCarthy}, {\it The three point Pick problem on the bidisk}, New York J. Math. 6 (2000), 227--236.

\bibitem{Agl4}  {\sc J.~Agler, J.E.~McCarthy}, {\it Distinguished varieties}, Acta Math. 194 (2005), no. 2, 133--153.

\bibitem{Aglbook} {\sc J.~Agler, J.E.~McCarthy}, {\it Pick interpolation and Hilbert function spaces}, Graduate Studies in Mathematics, 44. American Mathematical Society, Providence, RI, 2002.





\bibitem{Ama} \textsc{E.~Amar, P.~J.~Thomas}, \textit{A notion of extremal discs related to interpolation in the Ball}, Math. Ann., 300 (1994) 419--433.

\bibitem{Ama1} {\sc E. Amar, P~.J.~Thomas}, {\it Finite interpolation with minimum uniform norm in $\mathbb C^n$}. J. Funct. Anal. 170 (2000), no. 2, 512--525.





\bibitem{Cole} {\sc B. Cole, K. Lewis, J. Wermer}, {\it Pick conditions on a uniform algebra and von Neumann inequalities}, J. Funct. Anal. 107 (1992), no. 2, 235--254.

\bibitem{Com 2000} \textsc{D. Coman}, \textit{The Pluricomplex Green Function with Two Poles of the unit Ball in $\mathbb C^n$}, 
Pacific J. Math. 194 (2000), 257-283.
 




\bibitem{Edi 1995} \textsc{A.~Edigarian}, \textit{On extremal mappings in complex ellipsoids}, Ann. Polon. Math. 62 (1995), 83--96.


\bibitem{Edi-Zwo 1998} \textsc{A. Edigarian, W. Zwonek}, \textit{Invariance of the Pluricomplex Green Function under Proper
Mappings with Applications}, Complex Variables Theory Appl. 35 (1998), 367-380.



\bibitem{Ham} {\sc R. Hamilton}, {\it Pick interpolation in several variables}, Proc. Amer. Math. Soc. 141 (2013), no. 6, 2097--2103. 







\bibitem{Jar-Pfl 2013} \textsc{M.~Jarnicki, P. Pflug}, \textit{Invariant Distances and Metrics in Complex Analysis - 2nd extended edition}, De Gruyter Expositions in Mathematics 9, 2013.

\bibitem{Kne 2015} \textsc{G.~Knese}, \textit{The von Neumann inequality for $3\times 3$ matrices}, Bull. London Math Soc., to appear.


\bibitem{Kos 2015} \textsc{\L. Kosi\'nski}, \textit{Three-point Nevanlinna-Pick problem in the polydisc}, Proc. London Math. Soc.,
111(4) (2015), 887-910.


\bibitem{Kos-Zwo 2015} \textsc{\L. Kosi\'nski, W. Zwonek}, \textit{Extremal holomorphic maps in special classes of domains}, Ann. Sc. Norm Sup. Pisa, XVI(1) (2016), 159-182.

\bibitem{Kos-Zwo 2016} \textsc{\L. Kosi\'nski, W. Zwonek}, \textit{Nevanlinna-Pick Problem and Uniqueness of Left Inverses in Convex Domains, Symmetrized Bidisc and Tetrablock}, J. Geom. Anal., DOI: 10.1007/s12220-015-9611-9.


\bibitem{Lem 1981} \textsc{L.~Lempert}, \textit{La m\'etrique de Kobayashi et la repr\'esentation des domaines sur la boule}, Bull. Soc. Math. France 109 (1981), no. 4, 427--474.


\bibitem{Loj 1991} \textsc{S.~\L ojasiewicz}, \textit{Introduction to Complex Analytic Geometry}, Birkh\"auser Verlag, 1991.





\bibitem{Pick} {\sc G.~Pick}, {\it \"Uber die Beschr\"ankungen analytischer Funktionen, welche durch vorgegebene Funcktionswerte bewirkt werden}, Math. Ann. 77 (1916) 7–23.



\bibitem{Rud 1980} \textsc{W. Rudin}, \textit{Function Theory in the Unit Ball of $\mathbb C^n$}, Springer, 1980.

\bibitem{Sar}  {\sc D. Sarason}, {\it Generalized Interpolation in $H^\infty$}, Trans. Amer. Math. Soc. 127, 1967, 179--203.
 


\bibitem{War 2015} \textsc{T. Warszawski}, \textit{(Weak) $m$-extremals and $m$-geodesics}, Complex Var. Elliptic Eq.,
60(8) (2015), 1077-1105.
 



\end{thebibliography}
\end{document}